\documentclass[reqno, 12pt]{amsart}
\pdfoutput=1
\makeatletter
\let\origsection=\section \def\section{\@ifstar{\origsection*}{\mysection}} 
\def\mysection{\@startsection{section}{1}\z@{.7\linespacing\@plus\linespacing}{.5\linespacing}{\normalfont\scshape\centering\S}}
\makeatother        

\usepackage{amsmath,amssymb,amsthm}
\usepackage{mathrsfs}
\usepackage{mathabx}\changenotsign
\usepackage{bbm}
 
\usepackage{xcolor}
\usepackage[backref]{hyperref}
\hypersetup{
    colorlinks,
    linkcolor={red!60!black},
    citecolor={green!60!black},
    urlcolor={blue!60!black}
}

\usepackage[open,openlevel=2,atend]{bookmark}

\usepackage[abbrev,msc-links,backrefs]{amsrefs} 
\usepackage{doi}

\renewcommand{\PrintDOI}[1]{\doi{#1}}

\usepackage[T1]{fontenc}
\usepackage{lmodern}
\usepackage[babel]{microtype}
\usepackage[english]{babel}

\usepackage{geometry}
\numberwithin{equation}{section}

\usepackage{enumitem}
\def\rmlabel{\upshape({\itshape \roman*\,})}

\makeatletter
\def\cute#1{\expandafter\@cute\csname c@#1\endcsname}
\def\@cute#1{\ifcase#1
	\or $C1$%
	\or $C2$%
\fi}

\AddEnumerateCounter{\cute}{\@cute}{2}

\makeatother

\def\clabel{\upshape({\itshape \cute*})}

\def\l{\ifmmode\ell\else\polishlcross\fi}

\makeatletter
\def\moverlay{\mathpalette\mov@rlay}
\def\mov@rlay#1#2{\leavevmode\vtop{   \baselineskip\z@skip \lineskiplimit-\maxdimen
   \ialign{\hfil$\m@th#1##$\hfil\cr#2\crcr}}}
\newcommand{\charfusion}[3][\mathord]{
    #1{\ifx#1\mathop\vphantom{#2}\fi
        \mathpalette\mov@rlay{#2\cr#3}
      }
    \ifx#1\mathop\expandafter\displaylimits\fi}
\makeatother

\DeclareFontFamily{U}  {MnSymbolC}{}
\DeclareSymbolFont{MnSyC}         {U}  {MnSymbolC}{m}{n}
\DeclareFontShape{U}{MnSymbolC}{m}{n}{
    <-6>  MnSymbolC5
   <6-7>  MnSymbolC6
   <7-8>  MnSymbolC7
   <8-9>  MnSymbolC8
   <9-10> MnSymbolC9
  <10-12> MnSymbolC10
  <12->   MnSymbolC12}{}
\DeclareMathSymbol{\powerset}{\mathord}{MnSyC}{180}

\let\epsilon=\varepsilon
\let\eps=\epsilon

\def\WW0{{\mathcal W}_0}
\def\C{\mathscr{C}}

\theoremstyle{plain}
\newtheorem{thm}{Theorem}[section]
\newtheorem{prop}[thm]{Proposition}
\newtheorem{clm}[thm]{Claim}
\newtheorem{lemma}[thm]{Lemma}
\newtheorem{cor}[thm]{Corollary}
\newtheorem{obs}[thm]{Observation}
\theoremstyle{definition}
\newtheorem{rem}[thm]{Remark}
\newtheorem{defi}[thm]{Definition}
\newtheorem{exmp}[thm]{Example}

\begin{document}

\begin{abstract}
	Given an integer $k \geq 2$ and a real number $\gamma\in [0, 1]$, 
	which graphs of edge density~$\gamma$ contain the largest 
	number of $k$-edge stars? For $k=2$ Ahlswede and Katona
	proved that asymptotically there cannot be more such stars than in a clique 
	or in the complement of a clique (depending on the value of $\gamma$).
	Here we extend their result to all integers $k\ge 2$.
\end{abstract}
\title[Maximum star densities]{Maximum star densities}

\author[Christian Reiher]{Christian Reiher}
\address{Fachbereich Mathematik, Universit\"at Hamburg,
  Bundesstra\ss{}e~55, D-20146 Hamburg, Germany}
\email{Christian.Reiher@uni-hamburg.de}

\author[Stephan Wagner]{Stephan Wagner}
\address{Department of Mathematical Sciences, 
	Stellenbosch University, Private Bag X1,
	Matieland 7602, South Africa}
\email{swagner@sun.ac.za}

\thanks{This material is based upon work
    supported by the National Research Foundation of South Africa under grant number 96236.}


\maketitle

\section{Introduction} 
\label{sec:intro}
Ahlswede and Katona~\cite{AK} wondered about the maximal number of cherries ($3$-vertex stars)
a graph may contain when the numbers of its vertices and edges are known. 
They obtained the complete answer to this question describing the extremal graphs for 
any pair consisting of a number of vertices and a number of edges. 
Roughly speaking, they found that the extremal 
graphs are in some sense close to being either cliques or complements of cliques. 

For their precise statement, we need to define quasi-complete graphs and quasi-stars. Given
nonnegative integers $n$ and $m$ with $0 \leq m \leq \binom{n}{2}$, the {\it quasi-complete}
graph with $n$ vertices and $m$ edges is constructed as follows:
\begin{itemize}
\item Write $m = \binom{a}{2} + b$, where $0 \leq b < a$.
\item Take a complete graph with $a$ vertices.
\item Add another vertex and attach it to $b$ of the previous vertices.
\item Add $n-a-1$ isolated vertices.
\end{itemize}
On the other hand, the {\it quasi-star} with $n$ vertices and $m$ edges can be obtained as the 
complement of the quasi-complete graph with $n$ vertices and $\binom{n}{2} - m$ edges.

\begin{thm}[\cite{AK}]\label{thm:ahlswede-katona}
Among all graphs $G$ with a given number $n$ of vertices and a given number $m$ of edges, the 
number of cherries is either maximized by the quasi-complete graph or the quasi-star.
\end{thm}

Moreover, Ahlswede and Katona showed that there exists a nonnegative integer $R$ 
such that the quasi-star is extremal for $0 \leq m \leq \frac12 \binom{n}{2} - R$ or 
$\frac12 \binom{n}{2} \leq m < \frac12 \binom{n}{2} + R$, while the quasi-complete graph is 
extremal for $\frac12 \binom{n}{2} - R < m \leq \frac12 \binom{n}{2}$ or 
$\frac12 \binom{n}{2} + R \leq m \leq \binom{n}{2}$. The value of $R$ depends on $n$ in a rather nontrivial way,
but it can be shown that $R = O(n)$, and the distribution of $\frac{R}{n}$ can be characterised---see \cites{AFNW,WW} for details.

Note that the number of cherries in a graph $G$ can be expressed as
\[
	\sum_{v \in V(G)} \binom{d(v)}{2}\,.
\]
By the handshake lemma, this is equal to
\begin{equation}\label{eq:degsq}
\frac12 \sum_{v \in V(G)} d(v)^2 - |E(G)|\,.
\end{equation}
Thus maximizing the number of cherries, given the number of vertices and edges, 
is equivalent to maximizing the sum of the squared degrees.

It is not difficult to deduce the following asymptotic version of 
Theorem~\ref{thm:ahlswede-katona}:

\begin{thm}\label{thm:ak-asymp}
Given nonnegative integers $n$ and $m$, the maximum number of cherries in a graph with $n$ 
vertices and $m$ edges is
\[
	\max \big( \gamma^{3/2}, \eta+(1-\eta)\eta^2 \big) \frac{n^3}{2} + O(n^2)\,,
\]
where $\gamma = m/\binom{n}{2}$ is the edge density and $\eta=1-\sqrt{1-\gamma}$.
\end{thm}

Our aim in this paper is to obtain an analogue of this asymptotic statement for arbitrary stars. 
While there is no exact identity of the form~\eqref{eq:degsq}, maximizing the number of $k$-leaf 
stars $S_k$ is still asymptotically equivalent to maximizing the $k^{\text{th}}$ degree moment: 
this is because the number of $k$-leaf stars (to be precise, the number of---not necessarily 
induced---subgraphs isomorphic to $S_k$) in a graph $G$ can be expressed as
\[
	\sum_{v \in V(G)} \binom{d(v)}{k}\,.
\]
Up to a factor $k!$, this is the number of injective homomorphisms from $S_k$ to $G$. The total number of homomorphisms from $S_k$ to $G$ is exactly the $k^{\text{th}}$ degree moment
\[
	\sum_{v \in V(G)} d(v)^k\,.
\]

Natural languages to talk about asymptotic graph theoretical statements are provided 
by Razborov's theory of flag algebra homomorphisms~\cite{RazF}, and by Lov\'{a}sz et al.'s 
theory of graphons, which is  nicely explained in Lov\'{a}sz' recent research 
monograph~\cite{Lov-LNGN}. 
For concreteness we shall work with the latter approach in this article. 

Recall that a {\it graphon} is a measurable symmetric function 
$W\colon [0,1]^2\longrightarrow [0,1]$; here the word ``measurable'' should be understood 
either in the sense of the Borel $\sigma$-algebra, or with respect to Lebesgue measurable 
sets, but it is immaterial for our concerns which one of these two interpretations one 
actually adopts; the demand that $W$ be ``symmetric'' just means that $W(x, y)=W(y, x)$ 
is required to hold for all $x, y\in [0,1]$. The space of all graphons is denoted by $\WW0$. 
The main quantities studied in extremal graph theory are so-called homomorphism densities, 
defined as follows: Given a graph $H$ and a graphon $W$ one stipulates
\[
	t(H, W)=\int_{[0,1]^{V(H)}}\prod_{ij\in E(H)}W(x_i, x_j)\prod_{i\in V(H)}\mathrm{d}x_i\,,
\]
and calls $t(H, W)$ the {\it homomorphism density} from $H$ to $W$. 

In order to formulate the result alluded to above we call a graphon $W$ a {\it clique} 
if modulo null sets it is the characteristic function of a quadratic set of the form 
$A\times A$ for some measurable $A\subseteq [0,1]$, and $W$ is said to be an {\it anticlique} 
if $1-W$ is a clique; further, we let the pictorial symbols ``$|$'' and ``$\wedge$'' denote 
the graphs on two vertices with one edge, and the graph on three vertices with two edges, 
respectively. Now what Ahlswede and Katona proved yields in the limit that among all 
graphons~$W$ for which $t(\,|\,, W)$ has some fixed value, those for which $t(\wedge, W)$ 
is maximal are either cliques or anticliques. Thus
\[
	t(\wedge, W)\le \max\left(\gamma^{3/2}, \eta+(1-\eta)\eta^2\right)\,,
\]
where $\gamma=t(\,|\,, W)$ and $\eta=1-\sqrt{1-\gamma}$, is the best possible inequality 
in this regard; one may observe that the clique yields a larger value of $t(\wedge, \cdot)$ 
if $\gamma>\tfrac12$, while the anticlique is better if $\gamma<\tfrac12$; interestingly, 
if $\gamma=\tfrac12$, there are, up to weak isomorphism in the sense 
of~\cite{Lov-LNGN}*{Chap~7.3}, exactly two extremal graphons. 

The question to find for a fixed graph $H$ and a fixed $\gamma\in [0, 1]$ 
the {\it maximum} value of~$t(H, W)$ as $W$ varies through ${\mathcal W}_0$ 
under the constraint $t(\,|\,, W)=\gamma$ is, of course, interesting in general; 
we recall that if $H$ is a clique, the answer is 
well known by a theorem proved independently by Kruskal~\cite{Krusk} and by Katona~\cite{Kat}. 
In its full generality their theorem speaks about hypergraphs; in the $2$-uniform case it 
tells us that 
\[
	t(K_r, W)\le t(\,|\,, W)^{r/2}\,
\]
holds for all integers $r\ge 2$ and all graphons $W$. 

The opposite question about the {\it minimum} possible value of $t(H, W)$ for fixed 
$t(\,|\,,W)$ has also been studied in the literature. 
There are many such results for bipartite 
graphs~$H$ making partial progress on Sidorenko's conjecture, which states that
$$t(H, W) \geq t(\,|\,,W)^{|E(H)|}$$
holds for all bipartite graphs $H$, see e.g.~\cites{LogCalc, KLL, Szeg, CoLe} for some recent contributions.

For non-bipartite graphs, 
the answer only seems to be known when~$H$ is a clique. In this case the answer is given by 
the clique density theorem from~\cite{CDT}, which was proved earlier for 
triangles by Razborov~\cite{RazT} and for cliques of order four by Nikiforov~\cite{Niki-4}.

We shall prove in this article that replacing ``$\wedge$'' by the star $S_k$ 
with $k$ edges one can still get the same qualitative conclusion, while the case distinction 
on whether a clique or an anticlique is better will depend in a different manner on $\gamma$. 
A reason as to why the case $H=S_k$ should be easier than the general case is that the 
homomorphism density~$t(S_k, W)$ may be interpreted as the $k^{\text{th}}$ moment of 
the vertex degree function. Recall that associated with each graphon $W$ one has its degree 
function $d_W\colon[0,1]\longrightarrow [0,1]$ defined by 
$d_W(x)=\int_0^1 W(x, y)\;\mathrm{d} y$ for all $x\in[0,1]$; clearly 
\[
	t(S_k, W)=\int_0^1 d^k_W(x)\; \mathrm{d} x\,.
\]

So one may ask the more general question to bound $\int_0^1 F\bigl(d_W(x)\bigr) \mathrm{d} x$ 
from above for any given function $F\colon [0,1]\longrightarrow \mathbb{R}$. 
We shall identify in Section~\ref{sec:suff} a slightly artificial condition 
(see Definition~\ref{dfn:23} below) that, 
when imposed on $F$, guarantees that the answer will again be that the extremal graphons 
are either cliques or anticliques. In other words, this means that
\begin{equation*}\label{11}
	\int_0^1 F\bigl(d_W(x)\bigr)\; \mathrm{d} x\le 
	\max\left(\bigl(1-\sqrt{\gamma}\bigr)F(0)+\sqrt{\gamma}F(\sqrt{\gamma}), 
	(1-\eta)F(\eta)+\eta F(1)\right)
\end{equation*}
will hold for all graphons $W$ with $t(\,|\,, W)=\gamma$, where again $\eta=1-\sqrt{1-\gamma}$. 
The verification of this will occupy Section~\ref{sec:cute}. 
It is not entirely obvious that power functions $x\longmapsto x^k$ do indeed satisfy 
our condition; we shall confirm this in Section~\ref{sec:power}, 
thus obtaining the following:  

\begin{thm}\label{thm:main}
	Let $W$ be a graphon and let $k$ be a positive integer. 
	Set $\gamma=t(\,|\,, W)$ and $\eta=1-\sqrt{1-\gamma}$; we have the inequality
	\begin{equation}\label{eq:SkW}
		t(S_k, W)\le \max\left(\gamma^{(k+1)/2}, \eta+(1-\eta)\eta^k\right)\,.
	\end{equation}

	Moreover, for $k\ge 2$ there is some $\gamma_k\in (0, 1)$ such that this maximum is 
	$\eta+(1-\eta)\eta^k$ for $\gamma\in [0, \gamma_k]$ and $\gamma^{(k+1)/2}$
	for $\gamma\in [\gamma_k, 1]$.
\end{thm}

For $k \leq 30$, this was proven very recently by Kenyon, Radin, Ren and Sadun~\cite{KRRS} 
using a somewhat different approach; they also conjectured it to be true for arbitrary $k$. 
In another recent article, Nagy \cite{Nagy} obtained an analogous result for the density of 
another graph, namely the $4$-edge path: here it also turns out that cliques or anticliques 
(depending on $\gamma$) are extremal.

The following analogue of Theorem~\ref{thm:ak-asymp} is an immediate consequence of Theorem~\ref{thm:main}:

\begin{cor}
Given nonnegative integers $n$ and $m$ and an integer $k \geq 2$, the maximum number of copies of the star $S_k$ in a graph with $n$ vertices and $m$ edges is
$$\max \big( \gamma^{(k+1)/2}, \eta+(1-\eta)\eta^k \big) \frac{n^{k+1}}{k!} + O(n^k),$$
where $\gamma = m/\binom{n}{2}$ is the edge density and $\eta=1-\sqrt{1-\gamma}$.
\end{cor}

Likewise, we also have

\begin{cor}
Given nonnegative integers $n$ and $m$, the maximum of the $k^{\text{th}}$ degree moment $\sum_v d(v)^k$ in a graph with $n$ vertices and $m$ edges is
$$\max \big( \gamma^{(k+1)/2}, \eta+(1-\eta)\eta^k \big) n^{k+1} + O(n^k),$$
where $\gamma = m/\binom{n}{2}$ is the edge density and $\eta=1-\sqrt{1-\gamma}$.
\end{cor}

\begin{rem}
The quasi-complete graph and the quasi-star attain the bound asymptotically, but it is worth pointing out that they are not always the graphs for which the maximum number of copies of $S_k$ is attained. For example, if $k=3$, $n=13$ and $m=61$, then neither the quasi-complete graph nor the quasi-star contains the greatest number of copies of the star $S_3$: the quasi-complete graph has $1610$ copies, the quasi-star $1620$. Now consider the following graph:
\begin{itemize}
\item Start with a complete $11$-vertex graph and select three of its vertices, $v_1,v_2,v_3$.
\item Now add two more vertices $w_1,w_2$ and all six possible edges between $v_i$ and $w_j$ ($1 \leq i \leq 3$, $1 \leq j \leq 2$).
\end{itemize}
This graph has $13$ vertices and $61$ edges and contains $1622$ copies of $S_3$, which is in fact the maximum as we verified by means of a computer programme. An exhaustive search is possible, since the argument of \cite[Lemma 2]{AK} shows that the maximum is always attained by a graph whose vertices can be ordered as $v_1,v_2,\ldots,v_n$ in such a way that the following holds: if there is an edge between $v_i$ and $v_j$, then there is also an edge between $v_k$ and $v_l$ whenever $k \leq i$ and $l \leq j$, $k \neq l$.

With another computer search, we also found that the same graph has the greatest third degree moment (number of homomorphisms from $S_3$) at $13238$, as opposed to the quasi-complete graph and the quasi-star with $13202$ and $13172$ respectively.
\end{rem}

\section{A sufficient condition} 
\label{sec:suff}

Let us say that a measurable function $F\colon [0,1]\longrightarrow \mathbb{R}$ is {\it good} 
for a graphon $W$ if 
\[
	\int_0^1 F\bigl(d_W(x)\bigr) \mathrm{d} x\le 
	\max\left(\bigl(1-\sqrt{\gamma}\bigr)F(0)+\sqrt{\gamma}F(\sqrt{\gamma}), 
	(1-\eta)F(\eta)+\eta F(1)\right)
\]
holds, where $\gamma=t(\,|\,,W)$ and $\eta=1-\sqrt{1-\gamma}$. With this terminology we are 
interested in a condition for $F$ implying that it will be good for all graphons. 

A natural demand on the function $F$ is that it should be convex. Indeed this makes it 
more likely that quantities such as the left side of the above formula attain their 
maxima in fairly extreme situations, as we wish. Conversely, convexity already allows us 
to deal with an easy case.

\goodbreak

\begin{obs}\label{obs:21} 
	Convex function are good for all constant graphons.
\end{obs}

Indeed, if $W$ is constant always attaining the value $\gamma\in[0,1]$, 
then $t(\,|\,, W)=\gamma$ and
\[
	\int_0^1 F\bigl(d_W(x)\bigr) \mathrm{d} x=F(\gamma)\le 
	\bigl(1-\sqrt{\gamma}\bigr)F(0)+\sqrt{\gamma}F(\sqrt{\gamma}) 
\]
follows from Jensen's inequality. Notice that we could have verified
\[
	F(\gamma)\le (1-\eta)F(\eta)+\eta F(1)
\]
in the same way, for $\gamma=2\eta-\eta^2$.

Optimistically one might hope that all convex functions are good for all graphons, 
but unfortunately this is not the case, as the following construction demonstrates:
 
\begin{exmp}\label{ex:22}
	Let $F$ satisfy $F(0)=F\bigl(\tfrac15\bigr)=0$, $F\bigl(\tfrac35\bigr)= 1$, 
	$F(1)= 3$, and let $F$ be piecewise linear in between. Note that this function is convex. 
	Now look at $\gamma=\tfrac9{25}$, for which we have
	\[
		\max\left(\bigl(1-\sqrt{\gamma}\bigr)F(0)+\sqrt{\gamma}F(\sqrt{\gamma}), 
			(1-\eta)F(\eta)+\eta F(1)\right) = \tfrac35\,.
	\]

	Let $W$ be the characteristic function of $A\times A\cup A\times B\cup B\times A$, 
	where $A, B\subseteq [0,1]$ are disjoint. If the Lebesgue measures $y = \lambda(A)$ 
	and $z = \lambda(B)$ are chosen in such a way that $y+z\le 1$ and 
	$y^2+2yz=\gamma = \tfrac9{25}$, then the graphon $W$ satisfies $t(\,|\,, W)=\gamma$. 
	Thus we should have
	\[
		\int_0^1 F\bigl(d_W(x)\bigr) \mathrm{d} x = yF(y+z)+zF(y)\le \tfrac35
	\]
	for all possible choices of $y$ and $z$, but this is wrong, except for boundary cases.

	Indeed, this would mean that if $\tfrac15\le y\le \tfrac35$ and 
	$z=\tfrac{\gamma-y^2}{2y}$, we should have (note here that $\tfrac35\le y+z\le 1$)
	\[
		yF(y+z)+zF(y) = y \cdot (5y+5z-2)+z\cdot\tfrac{5y-1}{2}\le \tfrac35\,.
	\]
	It is plain that this fails e.g. for $y=\tfrac25$ and $z=\tfrac14$, for then the left side 
	of the inequality is~$\tfrac58$, which is greater than~$\tfrac35$. 
	One could show that any other choice of 
	$y\in\bigl(\tfrac15,\tfrac35\bigr)$ leads to a counterexample as well. 
	Also, one could modify $F$ slightly, replacing it by another function~$F^*$ that is
	differentiable infinitely often while $\|F-F^*\|_{\infty}$ is kept small. 
	In this way one can construct smoother counterexamples.
\end{exmp}

So we have to impose stronger conditions on $F$ than just convexity. The following definition, which might look artificial at first glance, provides us exactly with what we need:

\begin{defi}\label{dfn:23} 
	Let $\C$ denote the class of all twice differentiable convex functions 
	${F\colon [0,1]\longrightarrow\mathbb{R}}$ 
	satisfying the following two conditions.
	\begin{enumerate}[label=\clabel] 
	\item\label{it:C1} For all $a, b, y\in[0,1]$ with $a<y<b$ and
		\[
			\frac{F(b)-F(y)}{b-y}-\frac{F(y)-F(a)}{y-a}=F'(b)-F'(y)\,,
		\]
		we have
		\begin{align*}
			2&(b-y)\left[F(b)-(b-y)F'(b)+\tfrac12(b-y)^2F''(b)-F(y)\right] \\
			+&(y-a)\left[F(b)-(b-y)F'(y)+(b-y)^2F''(y)-F(y)\right]>0\,,
		\end{align*}
		and
	\item\label{it:C2} for all $a,b,y\in [0,1]$ with $a<y<b$ and
		\[
			\frac{F(b)-F(y)}{b-y}-\frac{F(y)-F(a)}{y-a}=F'(y)-F'(a)\,,
		\]
		we have
		\begin{align*}
			2&(y-a)\left[F(a)+(y-a)F'(a)+\tfrac12(y-a)^2F''(a)-F(y)\right] \\
			+&(b-y)\left[F(a)+(y-a)F'(y)+(y-a)^2F''(y)-F(y)\right]>0\,.
		\end{align*}
	\end{enumerate} 
\end{defi}

As we shall see in the next section, functions in $\C$ are good for 
all step graphons. Here, a {\it step graphon} is a symmetric function 
$W\colon [0,1]^2\longrightarrow [0,1]$ for which there exists a partition 
${\mathcal P}=\{P_1, \ldots, P_k\}$ of the unit interval into a finite number of 
measurable pieces such that~$W$ is constant on each rectangle of the 
form $P_i\times P_j$. 
It is known that the collection of all step graphons is dense in $\WW0$ with respect to 
the $L^1$-distance. So if $F$ is sufficiently well behaved and good for all 
step graphons, then by standard approximation arguments~$F$ is automatically good for 
all graphons. E.g., it suffices to assume that $F$ be continuously differentiable on $[0,1]$. 

\section{The main result on the class \texorpdfstring{$\C$}{\it C}} 
\label{sec:cute}

The principal goal of this section is to understand why the functions in $\C$ 
are good for all step graphons, cf. Proposition~\ref{prop:36} below. 
To prepare the proof of this assertion, we collect several lemmata about the 
functions in this class.  
The first of them informs us that~$\C$ is closed under several 
operations naturally appearing in our argument.

\begin{lemma}\label{lem:31} 
	Let $F\colon [0,1]\longrightarrow\mathbb{R}$ belong to $\C$.
	\begin{enumerate}[label=\rmlabel]
	\item\label{it:1} For all $A, B\in\mathbb{R}$, the function $x\longmapsto F(x)+A+Bx$ 
		belongs to $\C$ as well.
	\item\label{it:2} The function $G\colon [0,1]\longrightarrow\mathbb{R}$ given by 
		$x\longmapsto F(1-x)$ is also in $\C$.
	\item\label{it:3} For all real numbers $r$ and $s$ with $0\le r<s\le 1$, the function 
		$H\colon [0,1]\longrightarrow\mathbb{R}$ given by 
		${x\longmapsto F\bigl(r+(s-r)x\bigr)}$ is in $\C$.
	\end{enumerate}
\end{lemma}

\begin{proof} 
	The first part follows from the fact that neither the assumption nor the conclusion 
	of~\ref{it:C1} or~\ref{it:C2} change when a linear function is added to $F$. 
	Further, $G$ and $H$ are convex and satisfy the requested differentiability condition. 
	To see that $G$ satisfies~\ref{it:C1} for all numbers $a<y<b$ we apply~\ref{it:C2} for $F$ to 
	$1-b<1-y<1-a$ and vice versa. This shows that $G$ is indeed in $\C$. 
	To check similarly that $H$ satisfies~\ref{it:C1} or~\ref{it:C2} for all numbers 
	$a<y<b$, one applies the same property of $F$ to the numbers $r+(s-r)a<r+(s-r)y<r+(s-r)b$. 
\end{proof}
 
Our next steps are directed towards showing that the class of all graphons for which all
functions in $\C$ are good is likewise closed under some operations that occur later on. 
The first of these assertions, a rather direct consequence of part~\ref{it:2} from the 
foregoing lemma, does not carry the induction further by itself, but it will allow us to 
reduce one of two seemingly different cases to the other one. 
 
\begin{lemma}\label{lem:37}
	If all functions in $\C$ are good for a graphon $W$, then the same is true for
	the graphon $1-W$. 
\end{lemma} 

\begin{proof}
	Let $F\in\C$ be a function that we want to prove good for $1-W$. 
	Using the fact that the function~$G$ defined in~Lemma~\ref{lem:31}\ref{it:2} is good for $W$ 
	we find that the numbers 
	\[
	\gamma=t(\,|\,, 1-W)=1-t(\,|\,, W)
	\quad \text{ and } \quad 
	\eta=1-\sqrt{1-\gamma}
	\]
	satisfy
	\begin{align*}
		&\int_0^1 F\bigl(d_{1-W}(x)\bigr)\mathrm{d} x=
		\int_0^1 G\bigl( d_{W}(x)\bigr)\;\mathrm{d} x \\
		&\le
		\max\left(\bigl(1-\sqrt{1-\gamma}\bigr)G(0)+\sqrt{1-\gamma}G(\sqrt{1-\gamma}), 
		\sqrt{\gamma}G\bigl(1-\sqrt{\gamma}\bigr)+(1-\sqrt{\gamma}\bigr)G(1)\right) \\
		&=\max\left((1-\eta)F(\eta)+\eta F(1), 
		\bigl(1-\sqrt{\gamma}\bigr)F(0)+\sqrt{\gamma}F\bigl(\sqrt{\gamma}\bigr)\right)\,,
	\end{align*}
	as desired.
\end{proof}

The content of the next lemma is that one cannot construct ``$L$-shaped counterexamples'' 
for functions in $\C$ as in Example~\ref{ex:22}.
It is actually the only place in the entire proof where we really have to work in an 
essential way with~\ref{it:C1} and~\ref{it:C2}. Everything else follows by iterating this 
case by means of Lemma~\ref{lem:31} using convexity alone. 
 
\begin{lemma}\label{lem:32} 
	If $F\colon [0,1]\longrightarrow\mathbb{R}$ is in $\C$ and $x, y, z\in[0,1]$
	satisfy $x+y+z=1$, 
	then
	\[
		xF(0)+yF(y+z)+zF(y)\le 
		\max\left(\bigl(1-\sqrt{\gamma}\bigr)F(0)+\sqrt{\gamma}F(\sqrt{\gamma}), 
		(1-\eta)F(\eta)+\eta F(1)\right)\,,
	\]
	where $\gamma=y^2+2yz$ and $\eta=1-\sqrt{1-\gamma}$.
\end{lemma}

\begin{rem}
Observe here that $xF(0)+yF(y+z)+zF(y)$ is the value of 
$\int_0^1 F\bigl(d_W(x)\bigr) \mathrm{d} x$ for a graphon $W$ as constructed in 
Example~\ref{ex:22}.
\end{rem}

\begin{proof}[Proof of Lemma~\ref{lem:32}] 
	By Lemma~\ref{lem:31}\ref{it:1} we may assume $F(0)=0$ for simplicity. 
	If $\gamma=0$, then $y=z=0$, and if $\gamma=1$, then $y=1$ and $z=0$. 
	In both cases the claim is clear, so we may suppose $0<\gamma<1$ from now on. 
	As one easily confirms, the closed interval $C=\left[\eta, \sqrt{\gamma}\right]$ 
	is then non-trivial.
	Since 
	\[
		1-\gamma=(x+y+z)^2-(y^2+2yz)\ge (x+z)^2\,,
	\]
	we have $\eta\le 1-(x+z)=y$. Moreover $y^2\le \gamma$ entails $y\le \sqrt{\gamma}$, 
	so that altogether we get $y\in C$. Conversely, if for any $t\in C$ one sets 
	$z(t)=\tfrac{\gamma-t^2}{2t}$, then $z(t)\ge 0$ and
	\[
		t+z(t)=\frac{\gamma+t^2}{2t}=1-\frac{1-\gamma-(1-t)^2}{2t}\le 1\,,
	\]
	for which reason the numbers $1-t-z(t)$, $t$, and $z(t)$ satisfy the hypothesis on $x$, 
	$y$ and~$z$ in the statement of the lemma. Notice that $z(\eta)=1-\eta$ and
	$z(\sqrt{\gamma})=0$. Defining the function 
	$J\colon C\longrightarrow\mathbb{R}$ by 
	\[
		J(t)=tF\bigl(t+z(t)\bigr)+z(t)F(t)
	\]
	for all $t\in C$ we are to prove that $J(t)\le\max \left(J(\eta), J(\sqrt{\gamma})\right)$, 
	i.e., that $J$ attains its maximum at a boundary point of $C$. 
	If this failed, there would exist an interior point $t_0$ of~$C$ such that $J'(t_0)=0$ 
	but $J''(t_0)\le 0$. Since 
	\[
		z'(t)=-\frac{z(t)}{t}-1
	\]
	and thus 
	\begin{equation*}
		J'(t)=F\bigl(t+z(t)\bigr)-z(t)F'\bigl(t+z(t)\bigr)+z(t)F'(t)-
		F(t)-\frac{z(t)}{t} F(t)\,,
	\end{equation*}
	the equation $J'(t_0)=0$ can be rewritten as follows (recall that $F(0) = 0$ and $z(t_0) > 0$):
	\[
		\frac{F\bigl(t_0+z(t_0)\bigr)-F(t_0)}{z(t_0)}-\frac{F(t_0)-F(0)}{t_0}
		=F'\bigl(t_0+z(t_0)\bigr)-F'(t_0)\,.
	\]
	In other words, the numbers $0<t_0<t_0+z(t_0)$ are as $a,y,b$ in~\ref{it:C1}, and by the assumption that $F\in\C$ we have
	\begin{align*}
		2z(t_0)&\left[F\bigl(t_0+z(t_0)\bigr)-z(t_0)F'\bigl(t_0+z(t_0)\bigr)
		+\tfrac12z(t_0)^2F''\bigl(t_0+z(t_0)\bigr)-F(t_0)\right] \\ 
		+&t_0\left[F\bigl(t_0+z(t_0)\bigr)-z(t_0)F'(t_0)+z(t_0)^2F''(t_0)-F(t_0)\right]>0\,.
	\end{align*} 
	This rewrites as
	\[
		t_0z(t_0)J''(t_0)+\bigl(2z(t_0)+t_0\bigr)J'(t_0)>0\,,
	\]
contradicting our choice of $t_0$ as a point for which $J'(t_0) = 0$ and $J''(t_0) \le 0$. 
	This completes the proof of the lemma.
\end{proof}

To explain how the preceding lemma may actually be used in an inductive argument we 
introduce the following notation: given a graphon $W$ and a real number $\lambda\in[0,1]$ 
we define $[\lambda, W]$ to be the graphon satisfying
\[
	[\lambda, W](x, y)=
	\begin{cases}
		0 &\text{ if } 0\le x< \lambda \text{ or } 0\le y< \lambda, \\
		W\left(\frac{x-\lambda}{1-\lambda}, \frac{y-\lambda}{1-\lambda}\right)&  
		\text{ if } \lambda \le x\le 1  \text{ and } \lambda\le y\le 1.
	\end{cases}
\]

\begin{lemma} \label{lem:33}
	If $\lambda\in [0, 1]$ and the graphon $W$ has the property that all functions 
	in $\C$ are good for it, then the same applies to $[\lambda, W]$. 
\end{lemma}

\begin{proof}
	Let $F\in\C$ be any function that we want to prove good for $[\lambda, W]$. 
	By Lemma~\ref{lem:31}\ref{it:3} the function 
	$H\colon [0,1]\longrightarrow\mathbb{R}$ given by $H(x)=F\bigl((1-\lambda)x\bigr)$ 
	for all $x\in [0,1]$ is in $\C$. Thus it is good for~$W$, which tells us that
	\[
		\int_0^1 H\bigl(d_W(x)\bigr)\;\mathrm{d} x\le 
		\max\left(\bigl(1-\sqrt{\gamma}\bigr)H(0)+\sqrt{\gamma}H(\sqrt{\gamma}), 
		(1-\eta)H(\eta)+\eta H(1)\right)\,,
	\]
	where $\gamma=t(\,|\,, W)$ and $\eta=1-\sqrt{1-\gamma}$. 
	Since 
	\begin{align*}
		\int_0^1 F\bigl(d_{[\lambda,W]}(x)\bigr)\;\mathrm{d} x
		&=\lambda F(0)+\int_{\lambda}^1 F\left((1-\lambda)d_{W}
		\left(\frac{x-\lambda}{1-\lambda}\right)\right)\;\mathrm{d} x \\
		&=\lambda F(0)+(1-\lambda)\int_0^1 H\bigl(d_W(x)\bigr)\;\mathrm{d} x\,,
	\end{align*}
	it follows that either
	\[
		\int_0^1 F\bigl(d_{[\lambda,W]}(x)\bigr)\;\mathrm{d} x\le \lambda F(0)
		+(1-\lambda)\bigl(1-\sqrt{\gamma}\bigr)F(0)+(1-\lambda)
		\sqrt{\gamma}F\bigl((1-\lambda)\sqrt{\gamma}\bigr)\,, 
	\]
	or
	\[
		\int_0^1 F\bigl(d_{[\lambda,W]}(x)\bigr)\;\mathrm{d} x\le \lambda F(0)
		+(1-\lambda)(1-\eta)F\bigl((1-\lambda)\eta\bigr)+(1-\lambda)\eta F(1-\lambda)\,.
	\]
	In the former case the right side simplifies to 
	\[
		\bigl(1-\sqrt{\gamma'}\bigr)F(0)+
		\sqrt{\gamma'}F\bigl(\sqrt{\gamma'}\bigr)\,, 
	\]
	where $\gamma'=(1-\lambda)^2\gamma=t(\,|\,, [\lambda, W])$, meaning that $F$ is, 
	in particular, good for $[\lambda, W]$.

	\smallskip

	So from now on we may assume that the second alternative occurs. 
	Setting $x=\lambda$, $y=(1-\lambda)\eta$ and $z=(1-\lambda)(1-\eta)$ we thus get
	\[
		\int_0^1 F\bigl(d_{[\lambda,W]}(x)\bigr)\;\mathrm{d} x\le x F(0)+zF(y)+y F(y+z)\,.
	\]
	Since $y^2+2yz=(1-\lambda)^2(2\eta-\eta^2)=(1-\lambda)^2\gamma=\gamma'$, it follows 
	in view of Lemma~\ref{lem:32} that
	\[
		\int_0^1 F(d_{[\lambda,W]}(x))\mathrm{d} x\le 
		\max\left(\bigl(1-\sqrt{\gamma'}\bigr)F(0)+\sqrt{\gamma'}F(\sqrt{\gamma'}), 
		(1-\eta')F(\eta')+\eta' F(1)\right)\,,
	\]
	where $\eta'=1-\sqrt{1-\gamma'}$. This tells us that $F$ is indeed good for $[\lambda, W]$.
\end{proof}

A second construction we use is that of a graphon $[W, \lambda]$ defined for any real 
$\lambda\in[0,1]$ and graphon~$W$
by
\[
	[W, \lambda](x, y)=
	\begin{cases}
		W\left(\frac{x}{1-\lambda}, \frac{y}{1-\lambda}\right)&  \text{ if } 0\le x\le 
		1-\lambda \text{ and } 0\le y\le 1-\lambda\,, \\
		1 &\text{ if } 1-\lambda< x\le 1 \text{ or } 1-\lambda< y\le 1\,. \\
	\end{cases}
\]

\begin{lemma} \label{lem:34}
	If all functions in $\C$ are good for the graphon $W$ and $\lambda\in [0, 1]$,
	then all functions in $\C$ are good for $[W, \lambda]$ as well.
\end{lemma}

\begin{proof} 
	Since $[W, \lambda]$ is isomorphic to $1-[\lambda, 1-W]$, this follows from
	Lemma~\ref{lem:37} and Lemma~\ref{lem:33}. 
\end{proof}

Now we come to the main result of this section.

\begin{prop}\label{prop:36}
	Every function in $\C$ is good for every step graphon.
\end{prop}
  
\begin{proof}
	We prove this statement by contradiction. If it does not hold, let $W$ be a step graphon and
	$F\colon [0,1]\longrightarrow\mathbb{R}$ a function in $\C$ such that $W$ fails to be 
	good for $F$. Let~$W$ be a step function with respect 
	to the partition ${\mathcal P}=\{P_1, \ldots, P_k\}$ of the unit interval, 
	write $\alpha_i=\lambda(P_i)$ for each $i\in[k] =\{1,2,\ldots,k\}$, 
	and let $\beta_{ij}$ be the value 
	attained by $W$ on $P_i\times P_j$ for $i, j\in[k]$. Let $T$ denote the number of 
	pairs $(i, j)\in[k]^2$ for which $\beta_{ij}\in\{0,1\}$. 
	We may assume that among all possibilities $W$ has been chosen in such a way that 
	$k$ is as small as possible and subject to this $T$ is as large as possible. 
	It is plain that the numbers $\alpha_1, \ldots, \alpha_k$ are positive under this assumption.

	\smallskip

	Defining $d_i=\sum_{j=1}^{k}\alpha_j\beta_{ij}$ for each 
	$i\in[k]$ we are to prove
	\[
		\sum_{i=1}^{k}\alpha_iF(d_i)\le 
		\max\left(\bigl(1-\sqrt{\gamma}\bigr)F(0)+\sqrt{\gamma}F(\sqrt{\gamma}), 
		(1-\eta)F(\eta)+\eta F(1)\right)\,,
	\]
	where $\gamma=\sum_{i=1}^{k}\alpha_i d_i$ and $\eta=1-\sqrt{1-\gamma}$. 
	Without loss of generality we may assume ${d_1\le d_2\le\ldots\le d_k}$. 
	Observation~\ref{obs:21} shows that $k\ge 2$.

	\begin{clm}\label{clm:2}
		If $1\le i\le k$, $1\le r<s\le k$, and $\beta_{ir}>0$, then $\beta_{is}=1$. 
	\end{clm}

	\begin{proof}
		Assume $\beta_{ir}>0$ and $\beta_{is}<1$. We construct a new graphon $W'$ from $W$ with the same edge density 
by decreasing the value on $P_i \times P_r$ and $P_r \times P_i$ and increasing the value on $P_i \times P_s$ and $P_s \times P_i$. In order for these contributions to cancel, the changes need to be proportional to $\alpha_s$ and $\alpha_r$ respectively, and an additional factor of $2$ is needed in case $i=r$ or $i=s$, since then $P_i \times P_r$ and $P_r \times P_i$ ($P_i \times P_s$ and $P_s \times P_i$, respectively) coincide. 

Formally, define a step function $Q$ with respect to $\mathcal P$ as follows: let $\delta_{ij}$ denote the Kronecker delta, 
		and set, for $x \in P_m$ and $y \in P_n$,
		\[
			Q(x,y)=
			\begin{cases}
				-(1+\delta_{ir})\alpha_s &\text{ if  $\{m,n\}=\{i,r\}$}\,, \\
				(1+\delta_{is})\alpha_r &\text{ if  $\{m,n\}=\{i,s\}$}\,, \\
				0 &\text{ otherwise.}
			\end{cases}
		\]
		Let $\eps\ge 0$ be maximal such that 
		$W'=W+\eps\,Q$ is still a graphon, i.e.~maps to the interval $[0,1]$. 
		By our assumptions on $\beta_{ir}$ and $\beta_{is}$, $\eps$ is positive, 
		and the maximality of $T$ implies that $F$ is good for $W'$ (observe that $W'$ 
		is identically $0$ or $1$ on at least $T+1$ of the sets $P_i \times P_j$ 
		by construction). Moreover
		\[
			\int_{[0,1]^2}Q(x, y)\mathrm{d}x \mathrm{d}y=\alpha_i\alpha_r\alpha_s
			\bigl((1+\delta_{is})(2-\delta_{is})-(1+\delta_{ir})(2-\delta_{ir})\bigr)=0\,,
		\]
		whence we have $t(\,|\,,W')=t(\,|\,, W)$. So to derive the desired contradiction 
		we just need to check that
		\[
			\int_0^1 F\bigl(d_{W}(x)\bigr)\;\mathrm{d} x\le \int_0^1 
			F\bigl(d_{W'}(x)\bigr)\;\mathrm{d} x\,. 
		\]
		For $j\in [k]$ we let $d'_j$ denote the value attained by $d_{W'}(x)$ for 
		$x\in P_j$. Clearly $d'_j=d_j$ holds for all $j\not\in\{i,r,s\}$. Further
		\[
			d'_r-d_r=-(1+\delta_{ir})\alpha_i\alpha_s\eps
			+\delta_{ir}(1+\delta_{is})\alpha_i\alpha_s\eps=-\alpha_i\alpha_s\eps\,,
		\]
		and similarly
		\[
			d'_s-d_s=+\alpha_i\alpha_r\eps\,.
		\]
		Finally, if $i\not\in\{r,s\}$, then $d'_i=d_i$. So altogether we get indeed
		\begin{align*}
			\int_0^1 F\bigl(d_{W'}(x)\bigr)\;&\mathrm{d} x
			-\int_0^1 F\bigl(d_{W}(x)\bigr)\;\mathrm{d} x 
			=\sum_{j=1}^{k}\alpha_j\bigl(F(d'_j)-F(d_j)\bigr)\\
			=&\alpha_s\bigl(F(d_s+\alpha_i\alpha_r\eps)
			-F(d_s)\bigr)+\alpha_r\bigl(F(d_r-\alpha_i\alpha_s\eps)-F(d_r)\bigr) \\
			\ge&\alpha_i\alpha_r\alpha_s\eps\bigl(F'(d_s)-F'(d_r)\bigr)\ge 0 
		\end{align*}
		by the convexity of $F$ and because $d_s\ge d_r$. This proves Claim~\ref{clm:2}.
	\end{proof}
	
	\goodbreak
	
	\begin{clm} 
		$\beta_{1k}>0$. 
	\end{clm}

	\begin{proof}
		If this does not hold, then $\beta_{1k}=0$ and the previous claim entails $\beta_{1i}=0$ 
		for all $i\in [k-1]$. It follows that there is a step graphon $W'$ with 
		$k-1$ steps such that $W$ is isomorphic to~$[\alpha_1, W']$. 
		Due to the minimality of $k$ all functions in $\C$ are good for $W'$ and 
		by Lemma~\ref{lem:33} the same applies to the graphon $W$, contrary to its choice.  
	\end{proof}

	\begin{clm}
		$\beta_{1k}<1$.
	\end{clm}

	\begin{proof}
		If we had $\beta_{k1}=1$, then Claim~\ref{clm:2} would imply $\beta_{ki}=1$ 
		for all $i$ with $2\le i\le k$. So some step graphon $W'$ has the property 
		that $[W', \alpha_k]$ is isomorphic to $W$, which yields a contradiction 
		via Lemma~\ref{lem:34}. 
	\end{proof}

	So we must have $0<\beta_{1k}<1$. The conclusions drawn from Claim~\ref{clm:2} in the 
	two previous proofs are still valid, i.e., we	have $\beta_{1i}=0$ for all $i\in [k-1]$ 
	and  $\beta_{jk}=1$ for all $j$ with $2\le j\le k$. Divide $P_k$ into two measurable subsets
	$Q_k$ and $Q_{k+1}$ satisfying $\lambda(Q_k)=(1-\beta_{1k})\alpha_k$ and, consequently, 
	$\lambda(Q_{k+1})=\beta_{1k}\alpha_k$. 
	Set $Q_i=P_i$ for $i\in[k-1]$ and ${\mathcal Q}=\{Q_1, \ldots, Q_{k+1}\}$. 
	Let $W'$ be the step graphon with respect to ${\mathcal Q}$ defined as follows: 
	for $x\in Q_i$ and $y\in Q_j$,
	\[
		W'(x,y)=
		\begin{cases}
			\beta_{ij} &\text{ if  $2\le i\le k$ and $2\le j\le k$}\,, \\
			0 &\text{ if  $i=1$ and $j\in[k]$ or vice versa}\,, \\
			1 &\text{ if  $i=k+1$ or $j=k+1$}\,.
		\end{cases}
	\]
	By the last two clauses $W'$ is isomorphic to 
	a graphon of the form 
	$\bigl[\bigl[\tfrac{\alpha_1}{1-\beta_{1k}\alpha_k},W''\bigr],\beta_{1k}\alpha_k\bigr]$ 
	for some graphon $W''$, and by the first clause $W''$ is a step graphon with $k-1$ steps.
	So Lemma~\ref{lem:33} and Lemma~\ref{lem:34} show that $F$ is good for $W'$. 
	Since $t(\,|\,, W')=\gamma$, this means
	\begin{align*}
		\sum_{i=1}^{k-1}\alpha_iF(d_i)&+\alpha_k\bigl((1-\beta_{1k})F(d')
		+\beta_{1k}F(d'')\bigr)\\
		& \le \max\left(\bigl(1-\sqrt{\gamma}\bigr)F(0)+\sqrt{\gamma}F(\sqrt{\gamma}), 
		(1-\eta)F(\eta)+\eta F(1)\right)
	\end{align*}
	for some real numbers $d'$ and $d''$ satisfying $(1-\beta_{1k})d'+\beta_{1k}d''=d_k$. 
	Now Jensen's inequality implies that $F$ is indeed good for $W$. 
\end{proof}

\section{Verifying the assumption for power functions}
\label{sec:power}
 
The only thing currently missing from a proof of Theorem~\ref{thm:main} is that we do not know yet 
that for $k\ge 2$ the function $x\longmapsto x^k$ is indeed in the class $\C$. 
To verify this is the main objective of the present section. 
Fortunately the first half is easy due to the following lemma:

\begin{lemma}\label{lem:41}
	Suppose that $F\colon [0,1]\longrightarrow\mathbb{R}$ is thrice continuously 
	differentiable and satisfies $F''(x)>0$ 
	as well as $F'''(x)\ge 0$ 
	for all $x\in (0,1)$. Then $F$ has the property~\ref{it:C1}.
\end{lemma}

\begin{proof}
	We show something stronger, namely that independent of any further hypothesis 
	all real numbers $a<y<b$ from the unit interval satisfy
	\begin{align*}
		2&(b-y)\left[F(b)-(b-y)F'(b)+\tfrac12(b-y)^2F''(b)-F(y)\right] \\
		+&(y-a)\left[F(b)-(b-y)F'(y)+(b-y)^2F''(y)-F(y)\right]>0\,.
	\end{align*}
	Notice that the convexity assumptions on $F$ imply $F(b)\ge F(y)+(b-y)F'(y)$ 
	and~$F''(y)>0$, for which reason the second square bracket is positive. Furthermore 
	the general version of the mean value theorem yields the existence of some 
	real $\xi\in(y, b)$ such that
	\[
		F(y)=F(b)+(y-b)F'(b)+\tfrac12(y-b)^2F''(b)+\tfrac16(y-b)^3F'''(\xi)\,.
	\]
	Hence the first square bracket is $\tfrac16(b-y)^3F'''(\xi)\ge 0$. 
\end{proof}

\begin{rem}
	One could formulate a similar statement obtaining~\ref{it:C2} from $F'''(x)\le 0$. 
	Due to the symmetry expressed in Lemma~\ref{lem:31}\ref{it:2} and its proof 
	this is, of course, not surprising.
\end{rem}

To handle the second half we will use the following inequality twice:

\begin{lemma}\label{lem:43}
	If $x\ge 1$ is a real number and $m\ge 0$ an integer, then
	\[
		\sum_{i=0}^{m}(m+1-i)(3i-m)x^i\ge 0\,.
	\]
\end{lemma}

\begin{proof}
	It is obvious that
	\[
		\sum_{i=0}^{m-1}(i+1)(m-i)(m+1-i)x^i\ge 0\,.
	\]
	Multiplying this by $x-1$ gives us the desired inequality after some simple manipulations.
\end{proof}

\begin{prop}\label{prop:44}
	For each integer $k\ge 2$ the function $x\longmapsto x^k$ is in $\C$.
\end{prop}

\begin{proof}
	Condition~\ref{it:C1} holds by Lemma~\ref{lem:41}, so it remains to deal
	with~\ref{it:C2}. 
	Omitting the condition $b\le 1$ we prove that if any nonnegative 
	real numbers $a<y<b$ satisfy
	\[
		\frac{b^k-y^k}{b-y}-\frac{y^k-a^k}{y-a}=k(y^{k-1}-a^{k-1})\,,
	\]
	then
	\begin{align*}
		2&(y-a)\left[a^k+k(y-a)a^{k-1}+\binom{k}{2}(y-a)^2a^{k-2}-y^k\right] \\
		+&(b-y)\left[a^k+k(y-a)y^{k-1}+k(k-1)(y-a)^2y^{k-2}-y^k\right]>0\,.
	\end{align*}

	Everything is homogeneous, so we may suppose $y=1$ for notational simplicity. 
	So we are given that
	\begin{align}\label{441}
		\sum_{i=1}^{k-1}b^i=\sum_{i=1}^{k-1}a^i+k(1-a^{k-1})\,.
	\end{align}
	Applying Lemma~\ref{lem:43} to $x=b$ and $m=k-2$ we get
	\[
		\sum_{i=0}^{k-2}(k-1-i)(3i+2-k)b^i\ge 0\,.
	\]
	We multiply by $\tfrac{b-1}2$ (which is positive, since $b > y = 1$) to infer that
	\[
		\frac{(k-2)(k-1)}{2}+\sum_{i=1}^{k-1}(3i+1-2k)b^i\ge 0\,.
	\]
	We write $3i+1-2k$ as $3(i+1-k)+(k-2)$, split the sum and divide by $3$ to obtain
	\begin{align}\label{442}
		\sum_{i=0}^{k-1}(k-1-i)b^i\le 
		\frac{(k-1)k}{2}+\frac{k-2}3\Biggl\{\sum_{i=1}^{k-1}b^i-(k-1)\Biggr\}\,.
	\end{align}
	If $a\ne 0$, we can apply Lemma~\ref{lem:43} to $m=k-2$ and $x=\tfrac1a$ and obtain, 
	upon multiplication with $a^{k-2}$, the inequality
	\[
		\sum_{i=0}^{k-2}(k-1-i)(3i+2-k)a^{k-2-i}\ge 0\,.
	\]
	This is certainly also true for $a=0$, so it holds unconditionally. 
	Reversing the order of summation we find
	\[
		\sum_{i=0}^{k-2}(i+1)(2k-4-3i)a^{i}\ge 0\,,
	\]
	which may be weakened to
	\[
		\sum_{i=0}^{k-2}(i+1)(2k+2-3i)a^{i}> 0\,.
	\]
	Multiplying by $\frac{1-a}{2}$ (which is positive, since $a < y = 1$) we get
	\[
		\sum_{i=0}^{k-2}(k+1-3i)a^{i}+\frac{(k-1)(k-8)}{2}a^{k-1} > 0\,.
	\]
	Since $\frac{(k-1)(k-8)}{2} < 2(k-1)(k-2)$, this can again be weakened to
	\[
		\sum_{i=0}^{k-2}(k+1-3i)a^{i}+2(k-1)(k-2)a^{k-1}>0\,,
	\]
	which in turn can be rearranged in the same way as~\eqref{442} to read
	\[
		\frac{2(k-2)}{3}\Biggl\{\sum_{i=0}^{k-1}a^i-ka^{k-1}\Biggr\}
		<\sum_{i=0}^{k-2}(k-1-i)a^{i}\,.
	\]
	In combination with~\eqref{441} and~\eqref{442} this yields
	\begin{align*}
		2\sum_{i=0}^{k-1}(k-1-i)b^i &
		\le (k-1)k+\frac{2(k-2)}3\Biggl\{\sum_{i=1}^{k-1}b^i-(k-1)\Biggr\}\\
		&=(k-1)k+\frac{2(k-2)}3\Biggl\{\sum_{i=0}^{k-1}a^i-ka^{k-1}\Biggr\} \\
		&<(k-1)k+\sum_{i=0}^{k-2}(k-1-i)a^{i}\,.
	\end{align*}
	Now we multiply again by $b-1$ and use~\eqref{441}, thus learning that
	\[
		2\Biggl\{\sum_{i=0}^{k-1}a^{i}-ka^{k-1} \Biggr\} <(b-1)
		\Biggl\{(k-1)k+\sum_{i=0}^{k-2}(k-1-i)a^{i} \Biggr\}\,.
	\]
	The left side contains $1-a$ as a factor:
	\[\sum_{i=0}^{k-1}a^{i}-ka^{k-1} = (1-a)\sum_{i=0}^{k-2} (i+1)a^i \,.\]
	So after a further weakening we find
	\[
		2(1-a)\Biggl\{\sum_{i=0}^{k-2}(i+1)a^{i}- \binom{k}{2} a^{k-2} \Biggr\} 
		<(b-1)\Biggl\{(k-1)k+\sum_{i=0}^{k-2}(k-1-i)a^{i} \Biggr\}\,.
	\]
	Now is a good moment to multiply by $(1-a)^2$, because this leads to
	\begin{align*}
		2(1-a)\left(1-a^k -k(1-a)a^{k-1}-\binom{k}{2}(1-a)^2a^{k-2}\right) \\
 		<(b-1)\left(a^k+k(1-a)+k(k-1)(1-a)^2-1\right)\,,
	\end{align*}
	which is exactly what we wanted to prove.
\end{proof}

\begin{proof}[Proof of Theorem~\ref{thm:main}]
	We begin by showing~\eqref{eq:SkW}.
	The case $k=1$ is clear, so suppose $k\ge 2$ from now on. 
	If $W$ happens to be a step graphon, the result follows from Proposition~\ref{prop:36}, 
	its main assumption being verified in Proposition~\ref{prop:44}. 
	Now the general case follows from the known facts that both sides of the inequality we seek 
	to prove depend in a manner on~$W$ that is continuous with respect to the cut norm, 
	and that the step graphons are dense in $\WW0$ with respect to the cut norm. 
	
	For the moreover-part we suppose $k\ge 2$ and observe that for $\gamma=0$
	we have $\gamma^{(k+1)/2}=\eta+(1-\eta)\eta^k=0$. Proceeding with the case 
	$\gamma\in (0, 1]$ we take $\eps\in [0, 1)$ with $\gamma=1-\eps^2$ and, consequently, 
	$\eta=1-\eps$. Clearly   
	\begin{equation}\label{eq:squrediff}
		\gamma^{k+1}-\bigl(\eta+(1-\eta)\eta^k\bigr)^2
		=
		(1-\eps)^{k+1}\left((1+\eps)^{k+1}-(1-\eps)^{1-k}-2\eps-\eps^2(1-\eps)^{k-1}\right)\,,
	\end{equation}
	where the second factor  
	\begin{equation}\label{eq:Qdef}
		Q(\eps)=(1+\eps)^{k+1}-(1-\eps)^{1-k}-2\eps-\eps^2(1-\eps)^{k-1}
	\end{equation}
	develops into the convergent Taylor series
	\[
		Q(\eps)=(k-1)\eps^2+(k-1)\eps^3-\sum_{i=4}^\infty q_i \eps^i
	\]
	with 
	\[
		q_i=\binom{k+i-2}{i}-\binom{k+1}{i}+(-1)^i\binom{k-1}{i-2}
	\]
	for all $i\ge 4$. In particular, $Q(\eps)>0$ holds for all sufficiently small 
	positive values of $\eps$ and together with $\lim_{\eps \to 1^-}Q(\eps)=-\infty$
	it follows that the equation $Q(\eps)=0$ has at least one solution in $(0, 1)$.
	We shall prove later that, actually, there is a unique such solution, say~$\eps_k$.
	
	By~\eqref{eq:squrediff} this will tell us that the numbers $\gamma_k=1-\eps_k^2$
	and $\eta_k=1-\sqrt{1-\gamma_k}=1-\eps_k$ satisfy
	\[
		\gamma_k^{(k+1)/2}=\eta_k+(1-\eta_k)\eta_k^k \,.
	\]
	Besides, the estimates $Q(\eps)>0$
	for $\eps\in [0, \eps_k]$ and $Q(\eps)<0$ for $\eps\in [\eps_k, 1)$ translate 
	into the claims we made about the right side of~\eqref{eq:SkW}.
	
	It remains to establish the uniqueness of $\eps_k$. The crucial point is that for 
	every integer~${i\ge 4}$ we have 
	\[
		q_i
		\ge 
		\binom{k+2}{i}-\binom{k+1}{i}-\binom{k-1}{i-2}
		=
		\binom{k}{i-1}+\binom{k-1}{i-3}\ge 0\,.
	\]

	Now assume that there would exist two real numbers $0<\eps_*<\eps_{**}<1$ with 
	\[
		Q(\eps_*)=Q(\eps_{**})=0 \,.
	\]
	Because of~\eqref{eq:Qdef} this yields 
	\begin{align}
		(k-1)+(k-1)\eps_{*}&=\sum_{i=2}^\infty q_{i+2} \eps_{*}^i \label{eq:eps*} \\
		\intertext{	as well as }
		(k-1)+(k-1)\eps_{**}&=\sum_{i=2}^\infty q_{i+2} \eps_{**}^i \label{eq:eps**} \,.
	\end{align}
	Owing to~\eqref{eq:eps*} we obtain
	\[
		(k-1)+(k-1)\eps_{**}
		<
		\bigl((k-1)+(k-1)\eps_{*}\bigr)\frac{\eps_{**}}{\eps_{*}}
		=
		\sum_{i=2}^\infty q_{i+2} \eps_{*}^{i-1}\eps_{**}
		\le
		\sum_{i=2}^\infty q_{i+2} \eps_{**}^i \,,
	\]
	which contradicts~\eqref{eq:eps**}.
	This concludes the proof of the uniqueness of $\eps_k$ and, hence, the proof
	of Theorem~\ref{thm:main}.
\end{proof}

\begin{rem}
	Direct calculations reveal that $\gamma_2=\frac 12$ and $\gamma_3=\frac 34$.
	Moreover, it can be shown that $\gamma_k=1-\frac{\alpha^2}{k^2}+O\bigl(\frac 1{k^3}\bigr)$,
	where $\alpha\approx 1.5936$ denotes the unique positive solution of the equation  
	$\frac{\alpha}2+e^{-\alpha}=1$.
\end{rem}

\section*{Acknowledgement}

We would like to thank D\'aniel Nagy for pointing us to the recent references~\cite{KRRS} 
and~\cite{Nagy}. Further thanks go to an anonymous referee for reading this article carefully
and making useful suggestions.

\end{document}